\documentclass[reqno]{amsart}

\pdfoutput=1
\usepackage{amsrefs}
\usepackage{hyperref,enumerate}
\usepackage{graphicx,multirow}

\newtheorem{theorem}{Theorem}
\newtheorem{lemma}[theorem]{Lemma}

\newtheorem{corollary}[theorem]{Corollary}
\newtheorem{problem}[theorem]{Problem}
\theoremstyle{definition}
\newtheorem{definition}[theorem]{Definition}

\theoremstyle{remark}
\newtheorem{remark}[theorem]{\bf Remark}

\numberwithin{equation}{section}
\numberwithin{theorem}{section}

\newcommand{\intav}[1]{\mathchoice {\mathop{\vrule width 6pt height 3 pt depth  -2.5pt
\kern -8pt \intop}\nolimits_{\kern -6pt#1}} {\mathop{\vrule width
5pt height 3  pt depth -2.6pt \kern -6pt \intop}\nolimits_{#1}}
{\mathop{\vrule width 5pt height 3 pt depth -2.6pt \kern -6pt
\intop}\nolimits_{#1}} {\mathop{\vrule width 5pt height 3 pt depth
-2.6pt \kern -6pt \intop}\nolimits_{#1}}}

\newcommand{\intavl}[1]{\mathchoice {\mathop{\vrule width 6pt height 3 pt depth  -2.5pt
\kern -8pt \intop}\limits_{\kern -6pt#1}} {\mathop{\vrule width 5pt
height 3  pt depth -2.6pt \kern -6pt \intop}\nolimits_{#1}}
{\mathop{\vrule width 5pt height 3 pt depth -2.6pt \kern -6pt
\intop}\nolimits_{#1}} {\mathop{\vrule width 5pt height 3 pt depth
-2.6pt \kern -6pt \intop}\nolimits_{#1}}}

\newcommand{\R}{\mathbb{R}}

\renewcommand{\P}[1]{{\mathbb{P}}\left[{#1}\right]}
\newcommand{\PP}[2]{{\mathbb{P}}\left[{#1}|{#2}\right]}
\newcommand{\E}[1]{{\mathbb{E}}\left[{#1}\right]}
\newcommand{\EE}[2]{{\mathbb{E}}\left[{#1}|{#2}\right]}

\makeindex

\begin{document}

\title[A generalized P\'{o}lya's Urn with graph based interactions]{A generalized P\'{o}lya's Urn with\\ graph based interactions}

\author[Michel Bena\"im]{Michel Bena\"im}
\address{Institut de Math\'ematiques, Universit\'e de Neuch\^atel, 11 rue \'Emile Argand, 2000 Neuch\^atel, Suisse.}
\email{michel.benaim@unine.ch}
\author[Itai Benjamini]{Itai Benjamini}
\author[Jun Chen]{Jun Chen}
\address{Weizmann Institute of Science, Faculty of Mathematics and Computer Science, POB 26, 76100, Rehovot, Israel.}
\email{itai.benjamini, jun.chen@weizmann.ac.il}
\author[Yuri Lima]{Yuri Lima}
\address{Department of Mathematics, University of Maryland, College Park, MD 20742, USA.}
\email{yurilima@gmail.com}


\date{\today}


\begin{abstract}
Given a finite connected graph $G$, place a bin at each vertex. Two bins are called a pair if they
share an edge of $G$. At discrete times, a ball is added to each pair of bins.
In a pair of bins, one of the bins gets the ball with probability proportional to its current
number of balls raised by some fixed power $\alpha>0$.
We characterize the limiting behavior of the proportion of balls in the bins.

The proof uses a dynamical approach to relate the proportion of balls to
a vector field. Our main result is that the limit set of the proportion of balls is
contained in the equilibria set of the vector field. We also prove that if $\alpha<1$
then there is a single point $v=v(G,\alpha)$ with non-zero entries such that the
proportion converges to $v$ almost surely.

A special case is when $G$ is regular and $\alpha\le 1$. We show e.g. that if $G$ is non-bipartite
then the proportion of balls in the bins converges to the uniform measure almost surely.
\end{abstract}

\maketitle

\tableofcontents

\section{Introduction and statement of results}

Let $G=(V,E)$ be a finite connected graph with $V=[m]=\{1,\ldots,m\}$ and $|E|=N$, and assume
that on each vertex $i$ there is a bin initially with $B_i(0)\ge 1$ balls. For a fixed parameter $\alpha>0$,
consider a random process of adding $N$ balls to these bins at each step, according to the following
law: if the numbers of balls after step $n-1$ are $B_1(n-1),\ldots,B_m(n-1)$, step $n$ consists
of adding, to each edge $\{i,j\}\in E$, one ball either to $i$ or to $j$, and the probability that the ball
is added to $i$ is
\begin{align}\label{transition probability}
\P{i\text{ is chosen among }\{i,j\}\text{ at step }n}=\dfrac{B_i(n-1)^\alpha}{B_i(n-1)^\alpha+B_j(n-1)^\alpha}\,\cdot
\end{align}

In this paper, we study the limiting behavior, as the number of steps grows, of the proportion of balls in the
bins of $G$. More specifically, let $N_0=\sum_{i=1}^m B_i(0)$ denote the initial total number of balls, let
\begin{align}\label{definition model}
x_i(n)=\dfrac{B_i(n)}{N_0+nN}\,,\ \ \ i\in[m],
\end{align}
be the proportion of balls at vertex $i$ after step $n$, and let $x(n)=(x_1(n),\ldots,x_m(n))$.
Call the point $(1/m,\ldots,1/m)$ the {\it uniform measure}.
The first result classifies the limiting behavior of $x(n)$ when $G$ is regular and $\alpha=1$.

\begin{theorem} \label{main thm alpha=1}
Let $G$ be a finite, regular, connected graph, and let $\alpha=1$.
\begin{enumerate}[(a)]
\item If $G$ is non-bipartite, then $x(n)$ converges to the uniform measure almost surely.
\item If $G$ is bipartite, then $x(n)$ converges to $\Omega$ almost surely.
\end{enumerate}
\end{theorem}

Above, $\Omega$ is a subset of the $(m-1)$-dimensional closed simplex defined as follows:
if $V=A\cup B$ is the bipartition of $G$, then
\begin{align}\label{definition Omega}
\Omega=\{(x_1,\ldots,x_m):\exists\,p,q\ge 0,p+q=2/m,\text{ s.t. }x_i=p\text{ on }A,x_i=q\text{ on }B\}.
\end{align}

Theorem~\ref{main thm alpha=1} includes the case of any finite complete graph. A complete graph with at
least three vertices is non-bipartite, thus the proportion of balls in the bins converges to
the uniform measure almost surely.

Theorem~\ref{main thm alpha=1} also includes the case of cycles: if the length is odd, then the
proportion converges almost surely to the uniform measure; if the length is even, then
the random process' limit set is contained in $\Omega$ almost surely.

Theorems~\ref{main thm alpha=1} is consequence, after finer
analysis, of a general result for any $G$ and any $\alpha>0$. We show that the random process
is a stochastic approximation algorithm, thus it is related to a vector field in the closed simplex, and
the limit set of the random process is contained in the equilibria set
of the vector field (see Theorem~\ref{theorem benaim}). Let $\Lambda$ denote such an equilibria set.

\begin{theorem}\label{main thm general graphs}
Let $G$ be a finite, connected graph, and let $\alpha>0$.
Then the limit set of $x(n)$ is contained in $\Lambda$ almost surely.
\end{theorem}

Call $G$ {\it balanced bipartite} if there is a bipartition $V=A\cup B$ with $\#A=\#B$.

\begin{corollary}\label{corollary nonbipartite}
Let $G$ be a finite, connected, not balanced bipartite graph, and let $\alpha=1$.
Then $\Lambda$ is finite and $x(n)$ converges to an element of $\Lambda$ almost surely.
\end{corollary}

Theorem \ref{main thm general graphs} also allows to characterize the
limiting behavior of the random process when $\alpha<1$ and $G$ is any graph.

\begin{theorem}\label{main thm alpha<1}
Let $G$ be a finite, connected graph, and let $\alpha<1$. Then there is
$v=v(G,\alpha)$ with non-zero entries such that $x(n)$ converges to $v$ almost surely.
\end{theorem}

The last result deals with star graphs. The {\it star graph} is a tree with $m$ vertices and $m-1$ leaves.

\begin{theorem}\label{main thm star graph}
Let $G$ be a finite star graph with at least three vertices, and let $m$ be the vertex of highest degree.
\begin{enumerate}[(a)]
\item If $\alpha\le 1$, then $x(n)$ converges to
\begin{align*}
\left(\frac{1}{m-1+(m-1)^{\frac{1}{1-\alpha}}}\,,\,\ldots\,,\,\frac{1}{m-1+(m-1)^{\frac{1}{1-\alpha}}}\,,\,
\frac{(m-1)^{\frac{1}{1-\alpha}}}{m-1+(m-1)^{\frac{1}{1-\alpha}}}\right)
\end{align*}
almost surely.
\item If $\alpha>1$, then with positive probability $x(n)$ converges to $(0,\ldots,0,1)$,
and with positive probability $x(n)$ converges to $\left(\frac{1}{m-1},\ldots,\frac{1}{m-1},0\right)$.
\end{enumerate}
\end{theorem}

A {\it motivation} for the model proposed above is: imagine there are 3 companies,
denoted by M, A and G. Each company sells two products.
M sells {\sf OS} and {\sf SE}, A sells {\sf OS} and {\sf SP}, G sells {\sf SE} and {\sf SP}.
Each pair of companies compete
on one product. The companies try to use their global size and reputation to boost sales.
A natural question is which company will sell more products in the long term. In this scenario,
one can easily see that the interacting relation among the three companies forms a triangular
network. On this triangle, a vertex represents a company and an edge represents a product.
Under further simplifications, our model describes in broad strokes the long-term evolution
of such competition.

Another motivation for the model comes from a repeated game in which agents improve their
skill by gaining experience. The interaction network between agents is modeled by a graph.
At each round a pair is competing for a ball. A competitor improves his skill with time,
and the number of balls in his bin represents his skill level.

There are several natural ways to generalize our model to capture more complex interactions,
but here we focus on the simplest setup. The model can be viewed as a class of graph based
evolutionary model, which has been studied in various fields, e.g. biology \cite{ohtsuki2006simple},
economics \cite{jackson2010social}, mathematics \cite{skyrms2000dynamic},
and sociology \cite{santos2006evolutionary}.

The classical {\it P\'olya's urn} contains balls of two colors, and at each step one ball is drawn
randomly, its color is observed, and a ball of the same color is added to the urn.
In our process the interactions occur among pairs of bins, and on each of them the
model evolves similar to the classical P\'{o}lya's urn. We therefore call it a {\it generalized
P\'{o}lya's urn with graph based interactions}. For other variations  on P\'{o}lya's urn,  see e.g.
\cite{chung2003generalizations,launay2011interacting,marsili1998self}.

Let us sketch the ideas used in this article. We show that the proportion of balls in the
bins is a \emph{stochastic approximation algorithm}, then it is related to a vector field in the closed simplex.
In our case the vector field is gradient-like: it has a strict Lyapunov function (see Definition
\ref{definition strict lyapunov function}), whose set of critical values has empty interior.
This is enough to prove Theorem~\ref{main thm general graphs}. Corollary \ref{corollary nonbipartite}
follows from concavity properties of the strict Lyapunov function.

The method described above is usually called a {\it dynamical approach},
because the vector field generates an autonomous ordinary differential equation
(ODE), and dynamical information on the ODE give information on the random process.
From now on, we make no distinction between the vector field and its ODE.

The dynamical approach restricts the possible limits of the random process.
Some equilibria are stable, and some are unstable. In order to exclude convergence to unstable
equilibria, we look at the random process itself. Using probabilistic arguments,
we give checkable conditions to guarantee that the random process has zero probability
to converge to unstable equilibria (see Lemma \ref{lemma zero probability unstable}).
This is enough to prove Theorems \ref{main thm alpha=1}, \ref{main thm alpha<1}
and \ref{main thm star graph}.

This work initiates a program to understand the limiting behavior of $x(n)$ and its
dependence on $G$ and $\alpha$. Several open problems are presented in the last section.

\section{Stochastic approximation algorithms}\label{section stochastic approximation algorithm}

A {\it stochastic approximation algorithm} is a discrete time stochastic process whose general form can be written as
\begin{equation} \label{stochastic approximation}
x(n+1)-x(n)=\gamma_n H(x(n),\xi(n))
\end{equation}
where $H:\R^m\times\R^m\to\R^m$ is a measurable function that characterizes the algorithm, $\{x(n)\}_{n\ge 0}\subset\R^m$
is the sequence of parameters to be recursively updated, $\{\xi(n)\}_{n\ge 0}\subset\R^m$ is a sequence of random inputs
where $H(x(n),\xi(n))$ is observable, and $\{\gamma_n\}_{n\ge 0}$ is a sequence of ``small" nonnegative scalar gains.
Such processes were first introduced in the early 50s on the works of Robbins and Monro~\cite{robbins1951stochastic}
and Kiefer and Wolfowitz~\cite{kiefer1952stochastic}.

Let us show that the random process defined by (\ref{definition model})
is a stochastic approximation algorithm. We start by specifically formulating the model. Write $i\sim j$ whenever
$\{i,j\}\in E$. Let $\mathcal F_n$ denote the sigma-algebra generated by the process up to step $n$,
and let $C_i(n)$ denote the number of balls added to $i$ at step $n$. For each neighboring pair of
vertices $i\sim j$, consider 0-1 valued random variables $\delta_{i\leftarrow j}(n+1),\delta_{j\leftarrow i}(n+1)$
such that $\delta_{i\leftarrow j}(n+1)+\delta_{j\leftarrow i}(n+1)=1$ and
\begin{align}\label{transition probability 2}
\EE{\delta_{i\leftarrow j}(n+1)}{\mathcal F_n}=\dfrac{B_i(n)^\alpha}{B_i(n)^\alpha+B_j(n)^\alpha}
=\dfrac{x_i(n)^\alpha}{x_i(n)^\alpha+x_j(n)^\alpha}\,\cdot
\end{align}
Also, assume that $\delta_{i\leftarrow j}(n+1)$ and $\delta_{i'\leftarrow j'}(n+1)$ are independent whenever
the edges $\{i,j\}$ and $\{i',j'\}$ are distinct. Thus
\begin{align}\label{expression C_i}
C_i(n+1)=\sum_{j\sim i}\delta_{i\leftarrow j}(n+1).
\end{align}

We want to show that $x(n)=(x_1(n),\ldots,x_m(n))$ satisfies a difference equation of the
form (\ref{stochastic approximation}). Observe that
\begin{eqnarray*}
x_i(n+1)-x_i(n)&=&\dfrac{B_i(n)+C_i(n+1)}{N_0+(n+1)N}-\dfrac{B_i(n)}{N_0+nN}\\
               &&\\
               &=&\dfrac{-Nx_i(n)+C_i(n+1)}{N_0+(n+1)N}\\
               &&\\
               &=&\dfrac{1}{\frac{N_0}{N}+(n+1)}\left(-x_i(n)+\frac{1}{N}C_i(n+1)\right)
\end{eqnarray*}
and so $x(n)$ satisfies (\ref{stochastic approximation}) with
\begin{align}\label{definition gamma_n}
\gamma_n=\dfrac{1}{\frac{N_0}{N}+(n+1)}\ \ ,\ \ \ \xi(n)=\frac{1}{N}\left(C_1(n+1),\ldots,C_m(n+1)\right)
\end{align}
and $H:\R^m\times\R^m\rightarrow\R^m$ defined by
\begin{align*}
H(x(n),\xi(n))=-x(n)+\xi(n).
\end{align*}

Conditioning on $\mathcal F_n$, $H$ has a deterministic component $x(n)$ and a random component
$\xi(n)$. Nevertheless, nothing can be said about converging properties of $\xi(n)$. To this
matter, we modify the above equation by decoupling $\xi(n)$ into its mean part and the so
called ``noise'' part, which has zero mean. If one manages to control the total error of the
noise term, the limiting behavior of $x(n)$ can be identified via the limiting behavior of
the new deterministic component.

\section{The dynamical approach}\label{section dynamical approach}

The dynamical approach is a method introduced by Ljung~\cite{ljung1977analysis}
and Kushner and Clark~\cite{kushner1978stochastic} to analyze stochastic
approximation algorithms. Formally, it says that recursive expressions of the form
(\ref{stochastic approximation}) can be analyzed via an autonomous ODE
\begin{equation} \label{vectorODE}
\frac{dx(t)}{dt}=\overline H(x(t))\,,
\end{equation}
where $\overline H(x)=\lim_{n\to\infty}\E{H(x,\xi(n))}$.

In this perspective, our stochastic approximation algorithm can be written as
\begin{align*}
x(n+1)-x(n)=\gamma_n\left\{(-x(n)+\EE{\xi(n)}{\mathcal F_n})+(\xi(n)-\EE{\xi(n)}{\mathcal F_n})\right\}.
\end{align*}
Denote $\xi(n)=(\xi_1(n),\ldots,\xi_m(n))$. Equations (\ref{transition probability 2}),
(\ref{expression C_i}) and (\ref{definition gamma_n}) imply
\begin{align*}
\EE{\xi_i(n)}{\mathcal F_n}=\dfrac{1}{N}\sum_{j\sim i}\EE{\delta_{i\leftarrow j}(n+1)}{\mathcal F_n}
=\dfrac{1}{N}\sum_{j\sim i}\dfrac{x_i(n)^\alpha}{x_i(n)^\alpha+x_j(n)^\alpha}\,\cdot
\end{align*}
Thus, defining $\{u_n\}_{n\ge 0}\subset\R^m$ by
\begin{align}\label{definition u_n}
u_n=\xi(n)-\EE{\xi(n)}{\mathcal F_n}
\end{align}
and $F=(F_1,\ldots,F_m)$ by
\begin{align}\label{definition F}
F_i(x_1,\ldots,x_m)=-x_i+\dfrac{1}{N}\sum_{j\sim i}\dfrac{x_i(n)^\alpha}{x_i(n)^\alpha+x_j(n)^\alpha}\,,
\end{align}
our random process takes the form
\begin{align}\label{definition system}
x(n+1)-x(n)=\gamma_n\left[F(x(n))+u_n\right].
\end{align}

The above expression is a particular case of a class of stochastic approximation algorithms studied
in \cite{benaim1996dynamical}, on which the behavior of the algorithm is related to
a weak notion of recurrence for the ODE: that of {\it chain-recurrence}. Under the assumptions of
Kushner and Clark lemma~\cite{kushner1978stochastic}, it is proved that the accumulation points of
$\{x(n)\}_{n\ge 0}$ are contained in the chain-recurrent set of the ODE.

If, furthermore, the system is gradient-like and the set of critical values of the strict Lyapunov
function has empty interior, then the accumulation points of $\{x(n)\}_{n\ge 0}$ are equilibria
of the ODE. See Theorem~\ref{theorem benaim} below.

In this section we introduce some definitions, state Theorem~\ref{theorem benaim}, and prove that
our model satisfies two of the three conditions of Theorem~\ref{theorem benaim}.

Let $U\subset\R^m$ be a closed set, and let $F:U\rightarrow\R^m$ be a continuous vector field
with unique integral curves.

\begin{definition}[Equilibrium point]
A point $x\in U$ is called an {\it equilibrium} if $F(x)=0$. $x$ is called {\it stable} if
all the eigenvalues of $JF(x)$ have negative real part, and it is called {\it unstable} if
one of the eigenvalues of $JF(x)$ has positive real part.
Let $\Lambda$ denote the set of all equilibrium points. We call it the {\it equilibria set}.
\end{definition}

\begin{definition}[Strict Lyapunov function]\label{definition strict lyapunov function}
A {\it strict Lyapunov function for $F$} is a continuous map $L:U\to\R$ which is strictly
monotone along any integral curve of $F$ outside $\Lambda$. In this case, we call $F$
{\it gradient-like}.
\end{definition}

\subsection{A limit set theorem}

The reason we can characterize the limit set of the random process via the
equilibria set of the vector field is due to results in
\cite{benaim1996dynamical,benaim1999dynamics,benaim2012strongly} which,
to our purposes,  are summarized as follows.

\begin{theorem}\label{theorem benaim}
Let $F:\R^m\to\R^m$ be a continuous gradient-like vector field with unique integral curves,
let $\Lambda$ be its equilibria set, let $L$ be a strict Lyapunov function,
and let $\{x(n)\}_{n\ge 0}$ be a solution to the recursion
\begin{align*}
x(n+1)-x(n)=\gamma_n\left[F(x(n))+u_n\right],
\end{align*}
where $\{\gamma_n\}_{n\ge 0}$ is a decreasing gain sequence\footnote{$\lim_{n\to\infty}\gamma_n=0$ and
$\sum_{n\ge 0}\gamma_n=\infty$.} and $\{u_n\}_{n\ge 0}\subset\R^m$. Assume that
\begin{enumerate}[(i)]
\item $\{x(n)\}_{n\ge 0}$ is bounded,
\item for each $T>0$,
\begin{align*}
\lim_{n\to\infty}\left(\sup_{\{k:0\le\tau_k-\tau_n\le T\}}\left\Vert\sum_{i=n}^{k-1}\gamma_i u_i\right\Vert\right)=0\,,
\end{align*}
where $\tau_n=\sum_{i=0}^{n-1}\gamma_i$, and
\item $L(\Lambda)\subset\mathbb R$ has empty interior.
\end{enumerate}
Then the limit set of $\{x(n)\}_{n\ge 0}$ is a connected subset of $\Lambda$.
\end{theorem}

\begin{proof}
See Theorem 1.2 of \cite{benaim1996dynamical} and Proposition 6.4 of \cite{benaim1999dynamics}.
\end{proof}

\subsection{The random process (\ref{definition system}) satisfies (i) and (ii) of Theorem~\ref{theorem benaim}}

Firstly, note that $\{\gamma_n\}_{n\ge 0}$ satisfies
\begin{align*}
\lim_{n\to\infty}\gamma_n=0\ \ \text{ and}\ \ \sum_{n\ge 0}\gamma_n=\infty.
\end{align*}
Of course, $\{x(n)\}_{n\ge 0}$ is bounded. It remains to check condition (ii). For that, let
\begin{align*}
M_n=\sum_{i=0}^n\gamma_i u_i.
\end{align*}
$\{M_n\}_{n\ge 0}$ is a martingale adapted to the filtration $\{\mathcal F_n\}_{n\ge 0}$:
\begin{align*}
\EE{M_{n+1}}{\mathcal F_n}=\sum_{i=0}^n\gamma_i u_i+\EE{\gamma_{n+1}u_{n+1}}{\mathcal F_n}=\sum_{i=0}^n\gamma_i u_i=M_n.
\end{align*}
Furthermore, because for any $n\ge 0$
\begin{align*}
\sum_{i=0}^n\EE{\Vert M_{i+1}-M_i \Vert^2}{\mathcal F_i}\le\sum_{i=0}^n\gamma_{i+1}^2\le\sum_{i\ge 0}\gamma_i^2<\infty\ \text{ a.s.},
\end{align*}
the sequence $\{M_n\}_{n\ge 0}$ converges almost surely to a finite random vector
(see e.g. Theorem 5.4.9 of \cite{durrett2010probability}). In particular,
it is a Cauchy sequence and so condition (ii) holds almost surely.

In order to apply Theorem~\ref{theorem benaim}, we will construct a strict Lyapunov function for
the ODE
\begin{eqnarray}\label{definition ODE}
\left\{
\begin{array}{rcl}
\dfrac{dv_1(t)}{dt}&=&-v_1(t)+\dfrac{1}{N} \displaystyle\sum_{j\sim 1}\frac{v_1(t)^\alpha}{v_1(t)^\alpha+v_j(t)^\alpha}    \\
&\vdots&\\
\dfrac{dv_m(t)}{dt}&=&-v_m(t)+\dfrac{1}{N} \displaystyle\sum_{j\sim m}\frac{v_m(t)^\alpha}{v_m(t)^\alpha+v_j(t)^\alpha}
\end{array}\right.
\end{eqnarray}
that satisfies condition (iii) of Theorem~\ref{theorem benaim}.

Before that, let us specify the domain of the vector field $F$. Fix $c<1/N$,
and let $\Delta$ be the set of $m$-tuples $(x_1,\ldots,x_m)\in\mathbb R^m$ such that:
\begin{enumerate}[(1)]
\item $x_i\ge 0$ and $\sum_{i=1}^m x_i=1$, and
\item $x_i+x_j\ge c$ for all $\{i,j\}\in E$.
\end{enumerate}
Clearly, $F:\Delta\rightarrow T\Delta$ is  Lipschitz. Moreover, we have

\begin{lemma} \label{compactdomain}
$\Delta$ is positively invariant under the ODE (\ref{definition ODE}).
\end{lemma}

\begin{proof}
Given $\{i,j\}\in E$,
\begin{eqnarray*}
\dfrac{d}{dt}(v_i+v_j)&=&-v_i+\dfrac{1}{N}\sum_{k\sim i}\dfrac{v_i^\alpha}{v_i^\alpha+v_k^\alpha}
-v_j+\dfrac{1}{N}\sum_{l\sim j}\dfrac{v_j^\alpha}{v_j^\alpha+v_l^\alpha}\\
&\ge&-(v_i+v_j)+\dfrac{1}{N}\left(\dfrac{v_i^\alpha}{v_i^\alpha+v_j^\alpha}+\dfrac{v_j^\alpha}{v_j^\alpha+v_i^\alpha}\right)\\
&=&-(v_i+v_j)+\dfrac{1}{N}\cdot
\end{eqnarray*}
If $v$ belongs to the boundary of $\Delta$, there exists some $\{i,j\}\in E$ such that $v_i+v_j=c$, and then
\begin{align*}
\dfrac{d}{dt}(v_i+v_j)\ge -(v_i+v_j)+\dfrac{1}{N}=-c+\dfrac{1}{N}>0,
\end{align*}
which means that $F$ points inward on the boundary of $\Delta$. This proves that
condition (2) is preserved.
\end{proof}

\section{The general case: proof of Theorem~\ref{main thm general graphs}}\label{section main thm general graphs}

Let $L:\Delta\rightarrow\R$ be given by
\begin{align}\label{definition L}
L(v_1,\ldots,v_m)=-\sum_{i=1}^m v_i+\dfrac{1}{\alpha N}\sum_{\{i,j\}\in E}\log{(v_i^\alpha+v_j^\alpha)}.
\end{align}

\begin{lemma}\label{lemma lyapunov function}
$L$ is a strict Lyapunov function for $F$.
\end{lemma}

\begin{proof}
We have
\begin{align}\label{expression dL/dv_i}
\dfrac{\partial L}{\partial v_i}=-1+\dfrac{1}{\alpha N}\sum_{i\sim j}\dfrac{\alpha v_i^{\alpha-1}}{v_i^\alpha+v_j^\alpha}
=-1+\dfrac{1}{N}
\sum_{i\sim j}\dfrac{v_i^{\alpha-1}}{v_i^\alpha+v_j^\alpha},
\end{align}
thus
\begin{align}\label{expression of velocity}
\dfrac{dv_i}{dt}=v_i\left(-1+\dfrac{1}{N}\sum_{i\sim j}\dfrac{v_i^{\alpha-1}}{v_i^\alpha+v_j^\alpha}\right)
=v_i\dfrac{\partial L}{\partial v_i}\cdot
\end{align}
If $v=(v_1(t),\ldots,v_m(t))$, $t\ge 0$, is an integral curve of $F$, then (\ref{expression of velocity}) implies
\begin{eqnarray*}
\dfrac{d}{dt}(L\circ v)=\sum_{i=1}^m\dfrac{\partial L}{\partial v_i}\dfrac{dv_i}{dt}
=\sum_{i=1}^m v_i\left(\dfrac{\partial L}{\partial v_i}\right)^2\ge 0.
\end{eqnarray*}
Also, the last expression is zero if and only if $v_i\left(\frac{\partial L}{\partial v_i}\right)^2=0$ for all
$i\in[m]$, which is equivalent to $F(v)=0$.
\end{proof}

Let $\Lambda\subset\Delta$ be the equilibria set of $F$. For each
$S\subset[m]$, let
\begin{align*}
\Delta_S=\{v\in\Delta:v_i=0\text{ iff }i\notin S\}
\end{align*}
denote the face of $\Delta$ determined by $S$. $\Delta_S$ is a manifold with corners,
positively invariant under the ODE (\ref{definition ODE}).

\begin{definition}
$v\in\Delta_S$ is an {\it $S$-singularity for $L$} if
\begin{align*}
\frac{\partial L}{\partial v_i}(v)=0\ \text{ for all }i\in S.
\end{align*}
\end{definition}

Let $\Lambda_S\subset\Delta_S$ denote the set of $S$-singularities for $L$.

\begin{lemma}\label{characterization of Gamma}
$\Lambda=\bigcup_{S\subset[m]}\Lambda_S$.
\end{lemma}

\begin{proof}
If $v\in\Delta_S$, then $\frac{dv_i}{dt}=v_i\frac{\partial L}{\partial v_i}=0$ for all $i\notin S$.
Thus $v\in\Lambda$ iff $\frac{\partial L}{\partial v_i}=0$ for all $i\in S$.
\end{proof}

To prove Theorem~\ref{main thm general graphs}, it remains to show that $L(\Lambda)$ has empty interior.
$L|_{\Delta_S}$ is a $C^\infty$ function, thus by Sard's theorem
$L(\Lambda_S)$ has zero Lebesgue measure, so $L(\Lambda)$ has zero Lebesgue measure
as well. In particular, it has empty interior.

\subsection{Proof of Corollary \ref{corollary nonbipartite}}\label{subsection corollary 2}

Each restriction $L|_{\Delta_S}$ is concave\footnote{This follows from a simple fact on convex functions.
Let $\psi:\mathbb R\rightarrow\mathbb R$ be an increasing, concave function, and let
$f:\mathbb R^k\rightarrow\mathbb R$ be a (strictly) concave function. Then $\psi\circ f$ is
also (strictly) concave.}.
We claim that $L|_{\Delta_S}$ is strictly concave. Let $u,v\in\Delta_S$ and
$c\in(0,1)$. If $L(cu+(1-c)v)=cL(u)+(1-c)L(v)$, then $u_i+u_j=v_i+v_j$ for every $\{i,j\}\in E$, i.e.
\begin{align*}
u_i-v_i=(-1)(u_j-v_j)\, ,\ \ \forall\, \{i,j\}\in E.
\end{align*}
The values of $u_i-v_i$ along any path in $G$ alternate between $u_1-v_1$ and $-(u_1-v_1)$:
\begin{eqnarray}\label{equality non concave}
u_i-v_i&=&\left\{\begin{array}{ll}
u_1-v_1&\text{ if the distance from }i\text{ to }1\text{ is even,}\\
-(u_1-v_1)&\text{ if the distance from }i\text{ to }1\text{ is odd}.\\
\end{array}\right.
\end{eqnarray}
Let $A$ be the vertices within even distance to $1$ and $B$ those within odd distance
to $1$. If $G$ is non-bipartite, then $A\cap B\not=\emptyset$, thus $u_1=v_1$.
If $G$ is bipartite, then $V=A\cup B$ is the bipartition.
By (\ref{equality non concave}),
\begin{align*}
0=\sum_{i=1}^m u_i-\sum_{i=1}^m v_i=(u_1-v_1)\#A-(u_1-v_1)\#B.
\end{align*}
Because $G$ is not balanced bipartite, $\#A\not=\#B$ and again $u_1=v_1$. In any case we have
$u_1=v_1$, i.e. $u=v$. Thus $L|_{\Delta_S}$ is strictly concave. By the same reasoning as in
the end of the proof of Theorem \ref{main thm general graphs}, the corollary is established.

\section{Non-convergence to unstable equilibria}\label{section lemma unstable equilibrium}

In this section we give checkable conditions to guarantee that the random process has zero
probability to converge to unstable equilibria of $F$ when $\alpha\le 1$. This is Lemma
\ref{lemma zero probability unstable}, and its proof is an adaptation of the
proof of Theorem 1.3 in \cite{pemantle1992vertex}.

The checkable condition is in terms of the partial derivatives of $L$. Before stating it,
let us make some general considerations. By (\ref{expression dL/dv_i}),
\begin{align*}
\dfrac{\partial L}{\partial v_1}=-1+\dfrac{1}{N}\sum_{i\sim 1}\dfrac{v_1^{\alpha-1}}{v_1^\alpha+v_i^\alpha}\cdot
\end{align*}
$\partial L/\partial v_1$ is finite if $\alpha=1$ or $v_1>0$, and infinite otherwise.
Define $\partial L/\partial v_1:\Delta\rightarrow \mathbb R\cup\{\infty\}$, which is
continuous. In particular, if $\partial L/\partial v_1(v)>0$ then it is positive in a
neighborhood of $v$. In the proof of Lemma \ref{lemma zero probability unstable}
we will make use of the lemma below.

\begin{lemma}\label{lemma control neighborhood of v}
Let $v\in\Delta$ with $v_1=0$ and $\partial L/\partial v_1(v)>3\delta$.
Then there exists a neighborhood $\mathcal N$ of $v$, an element $u\in\mathcal N$,
and $\varepsilon_0>0$ such that
\begin{enumerate}[(i)]
\item $\frac{\partial L}{\partial v_1}(u)>3\delta+\frac{m\varepsilon_0}{N}$, and
\item for all $w\in\mathcal N$ and $i\sim 1$ it holds
\begin{align*}
\dfrac{w_1^{\alpha-1}}{w_1^\alpha+w_i^\alpha}>\dfrac{u_1^{\alpha-1}}{u_1^\alpha+u_i^\alpha}-\varepsilon_0.
\end{align*}
\end{enumerate}
\end{lemma}

\begin{proof}
Firstly, assume that $\alpha=1$. $\partial L/\partial v_1$ is continuous,
so we can fix a neighborhood $\mathcal N$ of $v$ satisfying (i). Each
$w\in\overline{\mathcal N}\mapsto 1/(w_1+w_i)$, $i\sim 1$, is uniformly continuous,
thus (ii) also holds if $\mathcal N$ is small enough.

Now assume that $\alpha<1$. Again, we can fix a neighborhood $\mathcal N$ of $v$ satisfying (i).
In this case, the maps $L_i:w\in\overline{\mathcal N}\mapsto w_1^{\alpha-1}/(w_1^\alpha+w_i^\alpha)$,
$i\sim 1$, are not uniformly continuous. But they are convex\footnote{This
is consequence of two facts of convex functions. {\bf Fact 1.} If
$f:\mathbb R^k\rightarrow(0,\infty)$ is concave, then $1/f$ is convex. {\bf Fact 2.}
$(x,y)\in(0,\infty)^2\mapsto x^\alpha y^{1-\alpha}$ is concave.}.
Because $L_i|_{\overline{\mathcal N}\cap\{w_1=0\}}=\infty$, its minimum is attained
outside $\overline{\mathcal N}\cap\{w_1=0\}$. Thus we can choose a small neighborhood
$\mathcal V$ of $\overline{\mathcal N}\cap\{w_1=0\}$ that does
not contain the minima of none of the $L_i$. Now, $L_i|_{\overline{\mathcal N}\backslash\mathcal V}$
is uniformly continuous. Restrict $\mathcal N$ if necessary and choose any
$u\in\overline{\mathcal N}\backslash\mathcal V$. Thus if $w\in\overline{\mathcal N}\backslash\mathcal V$
then $L_i(w)> L_i(u)-\varepsilon_0$, and if $w\in\mathcal V$ then
$L_i(w)\ge\min L_i|_{\overline{\mathcal N}\backslash\mathcal V}> L_i(u)-\varepsilon_0$.
\end{proof}

\begin{lemma}\label{lemma zero probability unstable}
Let $G$ be a finite, connected graph,  $\alpha\le 1$ and $L$ as in (\ref{definition L}). Let
$v$ be an equilibrium with $v_1=0$. If $\partial L/\partial v_1(v)>0$, then
\begin{align}\label{equation non-convergence unstable}
\P{\lim_{n\to \infty}x(n)=v}=0.
\end{align}
\end{lemma}

Before embarking into the proof, let us explain why the conditions of the lemma are equivalent to
$v$ being an unstable equilibrium. Firstly, assume that $\alpha=1$. We look at the
jacobian matrix $JF(v)$:
$$
\frac{\partial F_i}{\partial v_j}=\left\{\begin{array}{ll}
v_i\dfrac{\partial ^2L}{\partial v_i\partial v_j}&\text{ if }i\sim j,\\
&\\
\dfrac{\partial L}{\partial v_i}+v_i\dfrac{\partial ^2L}{\partial v_i^2} &\text{if }i= j, \\
&\\
0&\text{otherwise}.\\
\end{array}\right.
$$
Without loss of generality, assume that $v\in \Delta_S$ with $S=\{k+1,\ldots, m\}$. Thus
\begin{align}\label{definition jacobian}
JF(v)=\left[
\begin{array}{cc}
A & 0\\
C & B \\
\end{array}
\right]
\end{align}
where $A$ is a $k\times k$ diagonal matrix with $a_{ii}=\partial L/\partial v_i$, $i\in[k]$.
The spectrum of $JF(v)$ is the union of the spectra of $A$ and $B$. With respect to the inner
product $(x,y)=\sum_{i=k+1}^m x_iy_i/v_i$, $B$ is self-adjoint and negative semidefinite
(by the concavity of $L$), hence the eigenvalues of $B$ are real and nonpositive.
Therefore, $JF(v)$ has one real positive eigenvalue if and only if one of $a_{ii}$ is positive.

When $\alpha<1$, $JF(v)$ is not defined, in particular because $a_{11}$ is not finite.
Nevertheless, $a_{11}$ explodes to infinity, which intuitively means that $v$ is
unstable.

\begin{proof}
Firstly, we claim that
\begin{equation}\label{infinitely many balls}
\P{\lim_{n\to \infty} B_1(n)=\infty}=1.
\end{equation}
This is easy: because $1\le B_i(n)\le N_0+nN$, we have
\begin{align*}\label{transition probability}
\P{1\text{ is chosen among }\{1,i\}\text{ at step }n+1}=\dfrac{B_1(n)^\alpha}{B_1(n)^\alpha+B_i(n)^\alpha}\ge\dfrac{1}{2(N_0+nN)}
\end{align*}
for every $i\sim 1$, and so (\ref{infinitely many balls}) follows from the Borel-Cantelli lemma.

Fix $B>0$ large enough (to be specified later), and define
\begin{align*}
\mathcal{Y}_n=\{x(k)\in\mathcal N,\forall\,k\ge n\}\cap\{B_1(n)>B\},\ \ n>0.
\end{align*}
By (\ref{infinitely many balls}),
$\{\lim_{n\to\infty}x(n)=v\}\subset \bigcup_{m\ge n}\mathcal Y_m$ for any $n>0$. Thus
the lemma will be proved if we show that
\begin{equation}\label{equation non-convergence unstable 2}
\P{\mathcal Y_n}=0\ \ \ \text{ for sufficiently large }n.
\end{equation}

Let $\delta>0$ and $\mathcal N$ and in Lemma \ref{lemma control neighborhood of v}.
For a fixed $n_0$, let $\mathcal G_n=\mathcal F_n\cap\mathcal Y_{n_0}$, and let $c>0$ such that
\begin{eqnarray}\label{definition c}
\left[1+\dfrac{\delta(1+2\delta)}{1+\frac{3}{2}\delta}\right]\dfrac{1}{1+\delta}=1+c.
\end{eqnarray}

We claim that if $B$ is large enough, then there is $n_0>0$ such that
\begin{equation}\label{claim coupling}
\EE{\log x_1((1+\delta)n)}{\mathcal G_n}\ge\log x_1(n)+\frac{1}{2}\log(1+c) \ \text{ for all }n>n_0.
\end{equation}

Before proving the claim, let us show how to conclude the proof of the lemma. By contradiction, assume that
(\ref{equation non-convergence unstable 2}) is not true for some $n>n_0$.
Define $T_k=(1+\delta)^k n$ and $X_k=\log x_1(T_k)$. By (\ref{claim coupling}),
\begin{eqnarray*}
\EE{X_{k+1}}{\mathcal G_n}=\EE{\EE{X_{k+1}}{\mathcal G_{T_k}}}{\mathcal G_n}\ge \EE{X_k}{\mathcal G_n}+\frac{1}{2}\log(1+c).
\end{eqnarray*}
By induction,
\begin{eqnarray*}
\EE{X_k}{\mathcal G_n}\ge X_0+\dfrac{k}{2}\log(1+c)\ge-\log \left(N_0+nN \right) +\dfrac{k}{2}\log(1+c)
\end{eqnarray*}
which is a contradiction, because the left hand side is bounded.

We now prove (\ref{claim coupling}). The proof uses a coupling argument and Chernoff bounds.
Let $t\in\{n+1,\ldots,(1+\delta)n\}$. Restricted to $\mathcal Y_{n_0}$, we have
\begin{eqnarray*}\label{estimation probability}
\P{1\text{ is chosen among }\{1,i\}\text{ at step }t}&=&\frac{B_1(t-1)^\alpha}{B_1(t-1)^\alpha+B_i(t-1)^\alpha}\\
&&\\
&\ge&\dfrac{B_1(n)}{N_0+(t-1)N}\cdot\frac{x_1(t-1)^{\alpha-1}}{x_1(t-1)^\alpha+x_i(t-1)^\alpha}\\
&&\\
&\ge&\dfrac{B_1(n)}{N_0+(t-1)N}\left(\frac{u_1^{\alpha-1}}{u_1^\alpha+u_i^\alpha}-\varepsilon_0\right).
\end{eqnarray*}
Define a family of independent Bernoulli random variables $\{E_{t,i}\}$, $t=n+1,\ldots,(1+\delta)n$, $i\sim 1$,
as follows
\begin{eqnarray*}
\P{E_{t,i}=1}=\dfrac{B_1(n)}{N_0+(t-1)N}\left(\dfrac{u_1^{\alpha-1}}{u_1^\alpha+u_i^\alpha}-\varepsilon_0\right).
\end{eqnarray*}
Now couple $\{E_{t,i}\}$ to our model: if $E_{t,i}=1$, then $1$ is chosen among $\{1,i\}$ at step $t$.
If $n_0$ is large enough\footnote{Because $\log(1+x)>\frac{x}{1+x}$ for small $x>0$, if $n_0$ is large then
\begin{eqnarray*}
N\sum_{t=n+1}^{(1+\delta)n}\frac{1}{N_0+(t-1)N}>\log\left(\dfrac{N_0+(1+\delta)nN}{N_0+nN}\right)>
\log\left(1+\dfrac{\delta}{1+\frac{1}{2}\delta}\right)>\dfrac{\delta}{1+\frac{3}{2}\delta}\,\cdot
\end{eqnarray*}}, then
\begin{eqnarray*}
\E{\sum_{n+1\le t\le(1+\delta)n\atop{i\sim 1}}E_{t,i}}&\ge&B_1(n)\left(\sum_{t=n+1}^{(1+\delta)n}\dfrac{1}{N_0+(t-1)N}\right)
\left(\sum_{i\sim 1}\dfrac{u_1^{\alpha-1}}{u_1^\alpha+u_i^\alpha}-m\varepsilon_0\right)\\
&>&B_1(n)\dfrac{\delta(1+3\delta)}{1+\frac{3}{2}\delta}\cdot
\end{eqnarray*}
By Chernoff bounds (see Corollary A.1.14 of \cite{alon2004probabilistic}), if $\varepsilon_1>0$
then there is $B_0$ large enough such that
\begin{eqnarray}\label{large deviation inequality}
\P{\sum_{n+1\le t\le(1+\delta)n\atop{i\sim 1}}E_{t,i}>B_1(n)\dfrac{\delta(1+2\delta)}{1+\frac{3}{2}\delta}}>1-\varepsilon_1
\end{eqnarray}
for every $B_1(n)>B_0$. Whenever (\ref{large deviation inequality}) holds, the coupling gives that
\begin{eqnarray}
B_1((1+\delta)n)-B_1(n)\ge\sum_{n+1\le t\le(1+\delta)n\atop{i\sim 1}}E_{t,i}>B_1(n)\dfrac{\delta(1+2\delta)}{1+\frac{3}{2}\delta}
\end{eqnarray}
and thus by (\ref{definition c}) we have
\begin{eqnarray*}
x_1((1+\delta)n)>x_1(n)\left[1+\dfrac{\delta(1+2\delta)}{1+\frac{3}{2}\delta}\right]\dfrac{1}{1+\delta}=x_1(n)(1+c).
\end{eqnarray*}
From(\ref{large deviation inequality}), it follows that
\begin{eqnarray}\label{high_probability_increase_3}
\PP{x_1((1+\delta)n)>x_1(n)(1+c)}{\mathcal G_n}>1-\varepsilon_1.
\end{eqnarray}
Because $x_1((1+\delta)n)>\frac{x_1(n)}{1+\delta}$ and $\varepsilon_1>0$ can be chosen arbitrarily small,
(\ref{high_probability_increase_3}) gives
\begin{eqnarray*}  \label{expectation_increase}
\EE{\log x_1((1+\delta)n)}{\mathcal G_n}&>&(1-\varepsilon_1)\log(x_1(n)(1+c))+\varepsilon_1\log\left(\frac{x_1(n)}{1+\delta}\right)\\
&>&\log x_1(n)+\frac{1}{2}\log(1+c),
\end{eqnarray*}
thus establishing (\ref{claim coupling}).
\end{proof}

\section{Regular graphs: proof of Theorem \ref{main thm alpha=1}}\label{section regular graphs}

Here, $G$ is a finite, $r$-regular, connected graph, and $\alpha=1$. Assume first that $G$ is
non-bipartite, and let $u\in\Delta$ be the uniform measure. By Corollary \ref{corollary nonbipartite},
$\Lambda$ is finite and $x(n)$ converges to an element of $\Lambda$.
Furthermore, $\#\Lambda_S\le 1$ for every $S\subset[m]$. It is easy to check that $u\in\Lambda_{[m]}$,
thus $\Lambda_{[m]}=\{u\}$. We will show that any other equilibrium satisfies the
conditions of Lemma \ref{lemma zero probability unstable}, in which case we conclude the proof
of Theorem~\ref{main thm alpha=1} (a).

Now assume that $G$ is bipartite, and that $V=A\cup B$ is the bipartition of $G$. Let
$\Omega$ be defined as in (\ref{definition Omega}). Every $v\in\Omega$
is an equilibrium: for $i\in A$
\begin{align*}
F_i(v)=-v_i+\dfrac{1}{N}\sum_{j\sim i}\dfrac{v_i}{v_i+v_j}=-p+\dfrac{rp}{N(p+q)}=0,
\end{align*}
and the same holds for $i\in B$. We will show that any other equilibrium satisfies
the conditions of Lemma \ref{lemma zero probability unstable} and thus is unstable.
This being proved, Theorem~\ref{main thm alpha=1} (b) is established.

Summarizing the above discussion, to prove Theorem~\ref{main thm alpha=1} we just need
to prove the lemma below.

\begin{lemma} \label{equilibria_stability alpha=1}
Let $G$ be a finite, regular, connected graph, and let $\alpha=1$.
\begin{enumerate}[(a)]
\item If $G$ is non-bipartite, then every element of $\Lambda\backslash\{u\}$ is unstable.
\item If $G$ is bipartite, then every element of $\Lambda\backslash\Omega$ is unstable.
\end{enumerate}
\end{lemma}

\begin{proof}
Let $v\in\Lambda_S$ satisfying either (a) or (b). By Lemma \ref{lemma zero probability unstable},
it is enough to show that $\partial L/\partial v_i>0$ for some $i\in[m]\backslash S$.
Since $\partial L/\partial v_i=0$ for $i\in S$, it suffices to show that
$\sum_{i=1}^m\frac{\partial L}{\partial v_i}>0$, i.e.
\begin{equation}\label{positive_sum}
-m+\dfrac{1}{N}\sum_{i=1}^m\sum_{j\sim i}\dfrac{1}{v_i+v_j}>0.
\end{equation}
We first claim that the above expression is nonnegative. For this, note that the summand
has $2N$ terms and, by the arithmetic-harmonic mean inequality,
\begin{equation}\label{keyinequality}
\left[\sum_{i=1}^m \sum_{j\sim i}\frac{1}{v_i+v_j}\right]\cdot\left[\sum_{i=1}^m \sum_{j\sim i}(v_i+v_j)\right]\ge (2N)^2,
\end{equation}
with equality if and only if
\begin{equation} \label{equalitycons}
v_i+v_j={\rm const.}\ ,\ \ \forall\, \{i,j\}\in E.
\end{equation}
Since $G$ is $r$-regular, $N=rm/2$ and
\begin{align*}
\sum_{i=1}^m \sum_{j\sim i}(v_i+v_j)=2r.
\end{align*}
So (\ref{keyinequality}) gives
\begin{align}\label{am-hm inequality}
\sum_{i=1}^m\sum_{j\sim i}\frac{1}{v_i+v_j}\ge\dfrac{(2N)^2}{2r}=Nm,
\end{align}
thus proving our claim.

If (\ref{positive_sum}) is not true, then (\ref{equalitycons}) holds.
Fix the vertex $1$ of $G$, and let $v_1=p$ and $v_i=q$ for every neighbor $i\sim 1$.
Thus the values of $v_i$ along any path in the graph $G$ alternate between $p$ and $q$, i.e.
\begin{eqnarray*}
v_i&=&\left\{\begin{array}{ll}
p&\text{ if the distance from }i\text{ to }1\text{ is even,}\\
q&\text{ if the distance from }i\text{ to }1\text{ is odd}.\\
\end{array}\right.
\end{eqnarray*}

If $G$ is non-bipartite, it has a cycle of odd length, then $p=q$ and $v=u$, and if $G$ is
bipartite then $v\in\Omega$. In both cases, we get a contradiction.
\end{proof}

\section{Proof of Theorem~\ref{main thm alpha<1}}

If $\alpha<1$, then the function $x\mapsto x^\alpha,x>0$, is strictly concave. Thus each restriction
$L|_{\Delta_S}$ is also strictly concave\footnote{Again, we are using that
if $\psi:\mathbb R\rightarrow\mathbb R$ is an increasing, concave function, and
$f:\mathbb R^k\rightarrow\mathbb R$ is a (strictly) concave function, then $\psi\circ f$ is
also (strictly) concave.}, so $L$ has at most one $S$-singularity.

If $S\not=[m]$, then $v_i=0$ and $\partial L/\partial v_i=\infty$ on
$\Lambda_S$. By Lemma \ref{lemma zero probability unstable}, $\Lambda_S$ consists of an
unstable equilibrium. Let $\Lambda_{[m]}=\{v\}$. Thus $v$ has non-zero entries and
$x(n)$ converges to $v$ almost surely.

\section{Star graphs: proof of Theorem~\ref{main thm star graph}}\label{section main thm star graph}

When $G$ is the star graph with $m$ vertices and $m$ is the vertex with higher degree, (\ref{definition ODE}) becomes
\begin{eqnarray}\label{definition ODE star graph}
\left\{
\begin{array}{rcll}
\dfrac{dv_i(t)}{dt}&=&-v_i(t)+\dfrac{1}{m-1}\cdot\dfrac{v_i(t)^\alpha}{v_i(t)^\alpha+v_m(t)^\alpha}\ ,&i\in[m-1],\\
&&&\\
\dfrac{dv_m(t)}{dt}&=&-v_m(t)+\dfrac{1}{m-1}\displaystyle\sum_{j=1}^{m-1} \frac{v_m(t)^\alpha}{v_m(t)^\alpha+v_j(t)^\alpha}\,\cdot&
\end{array}\right.
\end{eqnarray}

\noindent {\bf Case 1: $\alpha\le 1$.}\\

When $\alpha=1$, note that $(0,\ldots,1)\in\Lambda$. We will show that any $v\in\Lambda\backslash\{(0,\ldots,0,1)\}$
satisfies the conditions of Lemma \ref{lemma zero probability unstable}.
By the arithmetic-harmonic mean inequality,
\begin{align*}
\dfrac{\partial L}{\partial v_m}=-1+\dfrac{1}{m-1}\sum_{i=1}^{m-1}\dfrac{1}{v_i+v_m}
\ge-1+\dfrac{m-1}{1+(m-2)v_m}>0.
\end{align*}
Because $v\in\Lambda$, we also have $v_m=0$.\\

When $\alpha<1$, direct calculations show that
\begin{align*}
\left(\frac{1}{m-1+(m-1)^{\frac{1}{1-\alpha}}}\,,\,\ldots\,,\,\frac{1}{m-1+(m-1)^{\frac{1}{1-\alpha}}}\,,\,
\frac{(m-1)^{\frac{1}{1-\alpha}}}{m-1+(m-1)^{\frac{1}{1-\alpha}}}\right)
\end{align*}
is an equilibrium point in the interior of $\Delta$.  By concavity of $L$, it is the
unique equilibrium point in the interior of $\Delta$.
The result thus follows from Theorem \ref{main thm alpha<1}.\\

\noindent{\bf Case 2: $\alpha>1$.}\\

When $m=2$, our model is a class of generalized P\'{o}lya's urn. For simplicity,
we refer to this process as ``g-urn". It is known
(see e.g. Theorem 4.1 of \cite{hill1980strong}) that in this case
\begin{align*}
\P{\lim_{n\to\infty}x(n)=(0,1)}>0\ \text{ and }\ \P{\lim_{n\to\infty}x(n)=(1,0)}>0.
\end{align*}

Now assume $m>2$. Observe that, as events,
\begin{align*}
\left\{\lim_{n\to\infty}x(n)=\left(\frac{1}{m-1},\ldots,\frac{1}{m-1},0\right)\right\}
\supset\bigcap_{i=1}^{m-1}\bigcap_{n\ge 1}\{\delta_{i\leftarrow m}(n)=1\}.
\end{align*}
By a coupling argument, we can identify this last event to the following one: in $m-1$
independent g-urns, just one color of ball is added in each g-urn
since the beginning of the process. Rubin's Theorem
(see e.g. Theorem 3.6 of \cite{pemantle2007survey}) guarantees that the event
``just one color of ball is added to the g-urn since the beginning of the process''
has positive probability, and so
\begin{align*}
\P{\lim_{n\to\infty}x(n)=\left(\frac{1}{m-1},\ldots,\frac{1}{m-1},0\right)}> 0.
\end{align*}

To prove the other claim, first observe that
\begin{align*}
\left\{\lim_{n\to\infty}x(n)=\left(0,\ldots,0,1\right)\right\}\supset
\bigcap_{i=1}^{m-1}\bigcap_{n\ge 1}\{\delta_{i\leftarrow m}(n)=0\}.
\end{align*}
By a coupling argument, the term on the right hand side of the above inclusion has positive probability
(again by Rubin's Theorem). This concludes the proof of Theorem~\ref{main thm star graph}.

\begin{remark}\label{remark independent set}
Given a finite connected graph $G=(V,E)$, call $I\subset V$ an {\it independent set} if $\{i,j\}\not\in E$ for
$i,j\in I$. The proof of Theorem~\ref{main thm star graph} gives the following: if $\alpha>1$ and $I$ is
an independent set, then
\begin{align*}
\P{\lim_{n\to \infty}x_i(n)=0,\ \forall\,i\in I}>0.
\end{align*}
\end{remark}

\section{Variants of the model} \label{section variants of the model}

\subsection{Edges with different weight functions}

Let $G=(V,E)$ be a finite, connected graph. For each edge $\{i,j\}\in E$, let
$f_{\{i,j\}}:(0,1)\rightarrow(0,1)$. A variant of the model is described as follows.
Let  $x_1(n-1),\ldots,x_m(n-1)$ be the proportions of balls after step $n-1$. At step $n$,
for each edge $\{i,j\}\in E$ add one ball either to $i$ or $j$ with probability
\begin{align*}
\P{i\text{ is chosen among }\{i,j\}\text{ at step }n}=\frac{f_{\{i,j\} }(x_i(n-1))}{f_{\{i,j\} }(x_i(n-1))+f_{\{i,j\} }(x_j(n-1))}\,\cdot
\end{align*}
In other words, we replace the law of $\delta_{i\leftarrow j}(n)$ in (\ref{transition probability 2})
by the above one, defined in terms of the $f_{\{i,j\}}$'s.

If we assume that, for each $\{i,j\}\in E$, $ f_{\{i,j\}}(x)=x^{\alpha{\{i,j\}}}$ for some $\alpha{\{i,j\}}>0$,
then
\begin{align*}
L(v_1,\ldots,v_m)=-\sum_{i=1}^m v_i+\dfrac{1}{N}\sum_{\{i,j\}\in E}\frac{\log{\left(v_i^{\alpha\{i,j\}}+v_j^{\alpha\{i,j\}}\right)}}{\alpha{\{i,j\}}}
\end{align*}
is a strict Lyapunov function for this variant model.

If furthermore each $\alpha{\{i,j\}}<1$, then we can argue as in the proof of
Theorem~\ref{main thm alpha<1} and conclude that there exists
$v$ (depending on $G$ and $\alpha\{i,j\},\{i,j\}\in E$) with non-zero entries
such that $x(n)$ converges to $v$ almost surely.

\subsection{Hypergraph based interactions}

We can similarly define a variant of the model on hypergraphs. Let $G=(V,E)$ be an hypergraph,
where $V=[m]$ and $E\subset 2^V$, $|E|=N$. Let  $x_1(n-1),\ldots,x_m(n-1)$ be the proportions
of balls after step $n-1$. At step $n$, for each hyperedge $e\in E$ add one ball to one of its
vertices with probability
\begin{align*}
\P{i\text{ is chosen on hyperedge }e\text{ at step }n}=\dfrac{x_i(n-1)}{\displaystyle\sum_{j\in e}x_j(n-1)}\,\cdot
\end{align*}
Notice that when $G$ is the trivial hypergraph with only one hyperedge $[m]$, this variant is a P\'olya's urn
model with balls of $m$ colors. See for instance \S 4.2 of \cite{pemantle2007survey}.

A special case of this variant was considered in \cite{skyrms2000dynamic}.
The authors defined\footnote{Actually, they only considered the model on complete graphs.}
a model called ``Friends II'' in a graph $\widetilde G=(\{1,\ldots,m\},\widetilde E)$.
If we define a hypergraph $G=(V,E)$ whose vertices are the edges of $\widetilde G$ and whose
hyperedges are the sets of neighboring edges in $\widetilde G$, i.e.
$V=\widetilde E$ and $E=\{e_1,\ldots,e_m\}$ with $e_i=\{\{i,j\}\in\widetilde E\}$,
then ``Friends II'' in $\widetilde G$ is the same as our variant in the hypergraph $G$.

The autonomous ODE of this variant is
\begin{eqnarray}\label{definition ODE for hypergraphs}
\left\{
\begin{array}{rcl}
\dfrac{dv_1(t)}{dt}&=&-v_1(t)+\dfrac{1}{N} \displaystyle\sum_{e\in E\atop{1\in e}}\frac{v_1(t)}{\sum_{j\in e}v_j(t)}\\
&\vdots&\\
\dfrac{dv_m(t)}{dt}&=&-v_m(t)+\dfrac{1}{N} \displaystyle\sum_{e\in E\atop{m\in e}}\frac{v_m(t)}{\sum_{j\in e}v_j(t)}\,\cdot\\
\end{array}\right.
\end{eqnarray}
In this case,
\begin{align*}
L(v_1,\ldots,v_m)=-\sum_{i=1}^m v_i+\dfrac{1}{ N}\sum_{e\in E} \log\left({\sum_{i\in e}v_i}\right)
\end{align*}
is a strict Lyapunov function.

In addition to the above two variants, one can also consider a model in which the
number of balls added at each step is some process possibly depending on the outcome so far.

\section{The case $\alpha=1$ for regular bipartite graphs}\label{section eigenvalues bipartite}

This section is of independent interest, and its purpose is to provide a better understanding of
the vector field $F$ when $G$ is a finite, $r$-regular, bipartite, connected graph, and $\alpha=1$.
We prove that every $v$ in the interior of $\Omega$ is stable in any direction transverse to $\Omega$,
by looking at the jacobian matrix
\begin{align}\label{definition jacobian}
JF=\left[
\begin{array}{ccc}
\dfrac{\partial F_1}{\partial v_1}&\cdots& \dfrac{\partial F_1}{\partial v_m}\\
\vdots&\ddots &\vdots\\
\dfrac{\partial F_m}{\partial v_1}&\cdots& \dfrac{\partial F_m}{\partial v_m}
\end{array}
\right]
\end{align}
of the vector field $F=(F_1,\ldots,F_m)$ defined by (\ref{definition ODE}).
Because $v\in\Omega$ belongs to a line of singularities, $0$ is an eigenvalue of $JF(v)$.
We prove that

\begin{lemma}\label{lemma hyperbolic attractor}
Let $v\in {\rm int}(\Omega)$. Any eigenvalue of $JF(v)$ different from $0$ has negative real part,
and $0$ is a simple eigenvalue of $JF(v)$.
\end{lemma}

\begin{proof}
We explicitly calculate the entries $\partial F_i/\partial v_k$.
Let $v\in\Omega$ with $v_i=p$ for $i\in A$ and $v_i=q$ for $i\in B$.
We have five cases:\\

\noindent $\bullet$ $i=k\in A$:
\begin{align*}
\dfrac{\partial F_i}{\partial v_i}(v)=-1+\dfrac{1}{N}\sum_{j\sim i}\dfrac{v_j}{(v_i+v_j)^2}=-1+\dfrac{mq}{2}\cdot
\end{align*}

\noindent $\bullet$ $i=k\in B$: analogously,
\begin{align*}
\dfrac{\partial F_i}{\partial v_i}(v)=-1+\dfrac{mp}{2}\cdot
\end{align*}

\noindent $\bullet$ $i\sim k$ and $i\in A$:
\begin{align*}
\dfrac{\partial F_i}{\partial v_k}(v)=-\dfrac{v_i}{N(v_i+v_k)^2}=-\dfrac{mp}{2r}\cdot
\end{align*}

\noindent $\bullet$ $i\sim k$ and $i\in B$: analogously,
\begin{align*}
\dfrac{\partial F_i}{\partial v_k}(v)=-\dfrac{mq}{2r}\cdot
\end{align*}

\noindent $\bullet$ $i\not\sim k$: in this case, $\partial F_i/\partial v_k=0$.\\

If we label the vertices of $A$ from $1$ to $m/2$, the vertices of $B$ from $1$ to $m/2$, and if we
let $M=(m_{ij})$ be the $m/2\times m/2$ adjacency matrix of the edges connecting vertices of $A$ to
vertices of $B$ (i.e. $m_{ij}=1$ when the $i$-th vertex of $A$ is adjacent to the $j$-th vertex of $B$),
then $JF(v)$ takes the form
\begin{eqnarray*}
JF(v)=-I+\dfrac{m}{2r}
\left[
\begin{array}{ccc}
rqI&&-pM\\
&&\\
-qM^t&&rpI\\
\end{array}
\right]\cdot
\end{eqnarray*}
Letting $\mu=p/(p+q)$ and $\nu=q/(p+q)$, $JF(v)$ can be written as
\begin{eqnarray}\label{description jacobian matrix}
JF(v)=-I+\dfrac{1}{r}\left[
\begin{array}{ccc}
r\nu I&&-\mu M\\
&&\\
-\nu M^t&&r\mu I\\
\end{array}
\right]=:-I+\dfrac{1}{r}S.
\end{eqnarray}
Given a matrix $X$, let $\sigma(X)$ denote its spectrum. By (\ref{description jacobian matrix}),
\begin{align*}
\sigma(JF(v))=\dfrac{1}{r}\sigma(S)-1
\end{align*}
and so it is enough to estimate the set $\sigma(S)$. The lemma will follow once we prove that
\begin{enumerate}[(a)]
\item every element of $\sigma(S)$ is either real or has real part equal to $r/2$,
\item $r$ is the largest real eigenvalue of $S$, and
\item $r$ is a simple eigenvalue of $S$.
\end{enumerate}

Let's prove (a). Let $\lambda=a+bi\in \sigma(S)$. Because $r\mu,r\nu<r$,
we can assume that $\lambda\not=r\mu,r\nu$. Note that the matrix
\begin{eqnarray*}
S-\lambda I=\left[
\begin{array}{ccc}
(r\nu-\lambda)I&&-\mu M\\
&&\\
-\nu M^t&&(r\mu-\lambda)I\\
\end{array}
\right]
\end{eqnarray*}
is singular if and only if its Schur complement
\begin{align*}
(r\nu-\lambda)I-(-\mu M)(r\mu-\lambda)^{-1}I(-\nu M^t)=
\dfrac{\mu\nu }{r\mu-\lambda}\left[\dfrac{(r\mu-\lambda)(r\nu-\lambda)}{\mu\nu}I-MM^t\right]
\end{align*}
is singular. Because $MM^t$ is symmetric, all of its eigenvalues are real and so
\begin{align*}
\dfrac{(r\mu-\lambda)(r\nu -\lambda)}{\mu\nu}\in\R \Longrightarrow (r\mu-\lambda)(r\nu-\lambda)\in\R.
\end{align*}
Calculating the imaginary part of $(r\mu-\lambda)(r\nu-\lambda)$, it follows that
\begin{align*}
-rb+2ab=0 \Longrightarrow b=0 \text{ or } a=r/2,
\end{align*}
which proves (a).

To prove (b), let $\lambda\in\sigma(S)\cap\mathbb R$, say $Sx=\lambda x$, where
$x=(x_1,\ldots,x_m)\in\R^m\backslash\{0\}$. Letting
$x_i=\max\{|x_1|,\ldots,|x_m|\}$ for $i\in A$, we have
\begin{align*}
\lambda x_i=r\nu x_i-\mu\sum_{j\sim i}x_j\le r\nu x_i+\mu\sum_{j\sim i}x_i=rx_i
\end{align*}
and thus $\lambda\le r$. The same holds if $i\in B$.

It remains to prove (c). When $\lambda=r$, we have for $i\in A$
\begin{eqnarray*}
rx_i=r\nu x_i-\mu\sum_{j\sim i}x_j \Longrightarrow x_i=-\dfrac{1}{r}\sum_{j\sim i}x_j
\end{eqnarray*}
and similarly for $i\in B$. Thus the function $h:V\rightarrow\R$ defined by
\begin{eqnarray*}
h(i)&=&\left\{\begin{array}{rl}
x_i&\text{if }i\in A,\\
-x_i&\text{if }i\in B\\
\end{array}\right.
\end{eqnarray*}
is harmonic in $G$. By the maximum principle, $h$ is constant, and so $r$ is a simple eigenvalue of $S$.
\end{proof}

\section{Further questions}\label{section questions}

This work is part of a program to answer the following

\begin{problem}
Given a finite connected graph $G$ and $\alpha>0$, what is the limiting behavior of $x(n)$?
\end{problem}

Theorem \ref{main thm alpha<1} gives a full answer when $\alpha<1$. When $\alpha=1$, it is not
clear what to expect, because the properties of the graph should be taken into account.
We conjecture the following.\\

\noindent {\bf Conjecture.} Let $G$ be a finite, connected, not balanced bipartite graph,
and let $\alpha=1$. Then there is a single point such that $x(n)$ converges to it almost surely.\\

When $\alpha>1$ the question remains widely open, even when $G$ is a triangle. The uniform measure is always
an equilibrium. When $1<\alpha<4/3$, it is stable and thus $x(n)$
converges to it with positive probability. Also, by Remark \ref{remark independent set},
for any $i \in \{1,2,3\}$, $x_i(n)$ converges to zero with positive probability.
In general, we think there exists $\alpha_0=\alpha_0(G)$ such that when $\alpha>\alpha_0$
\begin{align*}
\P{\lim_{n\to\infty}x(n)\in\partial{\Delta}}=1.
\end{align*}
When $G$ is the star graph, item (b) of Theorem~\ref{main thm star graph} gives a partial answer
to the question.

Now turn attention to the special cases we considered, summarized in Table \ref{table special cases}.

\begin{table}[ht]
\renewcommand{\arraystretch}{1.5}
\begin{tabular}{|c|c|c|c|}\hline
$G$ & {\bf regular non-bipartite} & {\bf regular bipartite} & {\bf star graph}\\ \hline\hline
$\alpha<1$ & \multirow{2}{*}{uniform measure}& uniform measure & \multirow{2}{*}{$v(m,\alpha)$ of Theorem \ref{main thm star graph}}\\ \cline{1-1}\cline{3-3}
$\alpha=1$ & & $\Omega$ & \\ \hline
&&& probability $>0$ to \\
$\alpha>1$ & ? & ? & $(0,\ldots,0,1)$ and\\
&&& $\left(\frac{1}{m-1},\ldots,\frac{1}{m-1},0\right)$\\ \hline
\end{tabular}
\caption{The limiting behavior of $x(n)$.}\label{table special cases}
\end{table}

When $G$ is regular bipartite and $\alpha=1$, Theorem~\ref{main thm alpha=1} says
that the limit set of $x(n)$ is contained in $\Omega$. However, we do not know if
the limit exists. When the number of vertices is two, the model is the classical
P\'{o}lya's urn, and in this case it is known that $x(n)$ converges to a
point of $\Omega$ almost surely. See e.g.~\S 2.1 of \cite{pemantle2007survey}.

\begin{problem}
For a general regular bipartite graph and $\alpha=1$, does $x(n)$ converge
to a point of $\Omega$ almost surely?
\end{problem}

In Section \ref{section eigenvalues bipartite} we proved that every point in
the interior of $\Omega$ is stable in any direction transverse to $\Omega$.

\begin{problem}\label{problem 3}
In Theorems \ref{main thm alpha=1}, \ref{main thm general graphs} and \ref{main thm alpha<1},
what is the rate of convergence of $x(n)$ to its limit?
\end{problem}

This problem is related to the control of the eigenvalues of $JF$ on $\Lambda$, and
of quantitative estimates on the precision that $\{x(n)\}_{n\ge 0}$ shadows a real orbit of the ODE
associated to $F$. See e.g. \S 3.2 of \cite{pemantle2007survey}.

Another question of interest is the following

\begin{problem}
What is the correlation between the number of balls in the bins, as a function of $\alpha$ and of the
number of steps $n$, e.g. when $G$ is an Euclidean lattice?
\end{problem}

\begin{remark}
After the preparation of this manuscript, we learned that the Conjecture and Problem 11.2 were solved affirmatively
\cite{chen2013generalized}.
\end{remark}

\section{Acknowledgements}

The authors are thankful to Omri Sarig for valuable comments and suggestions.
M.B. is supported by Swiss National Foundation, grant 138242.
I.B. is the incumbent of the Renee and Jay Weiss Professorial Chair. J.C. is supported by the ISF.
During the preparation of this manuscript, Y.L. was a Postdoctoral Fellow at the Weizmann Institute of
Science, supported by the ERC, grant 239885. Y.L. is supported by the Brin Fellowship.

\bibliography{polya_urn_bib}

\end{document}